\documentclass[12pt]{amsart}
\usepackage{amsfonts, amsmath, amssymb, amscd, latexsym, graphicx}

\newtheorem{thm}{Theorem}[section]
\newtheorem{prop}[thm]{Proposition}

\newtheorem{example}[thm]{Example}

\newtheorem{lemma}[thm]{Lemma}

\begin{document}

\title{Metric properties of the braided Thompson's groups}

\author{Jos\'e Burillo}
\address{
Departament de Matem\'atica Aplicada IV, Universitat Polit\'ecnica
de Catalunya, Escola Polit{\`e}cnica Superior de Castelldefels,
08860 Castelldefels, Barcelona, Spain} \email{burillo@mat.upc.es}

\author{Sean Cleary}
\address{Department of Mathematics,
The City College of New York \& The CUNY Graduate Center, New York,
NY 10031} \email{cleary@sci.ccny.cuny.edu}

\thanks{The authors are grateful for the hospitality of the
Centre de Recerca Matem\`atica. 
The first author acknowledges support from MEC grant \#MTM2006-13544-C02.
The second author acknowledges support from
the National Science Foundation and from PSC-CUNY Research Award \#69034}



\date{\today}



\begin{abstract}
Braided Thompson's groups are finitely presented groups introduced
by Brin and Dehornoy which contain the ordinary braid groups $B_n$,
the finitary braid group $B_{\infty}$ and Thompson's group $F$ as
subgroups. We describe some of the metric properties of  braided
Thompson's groups and give upper and lower bounds for word length in
terms of the number of strands and the number of crossings in the
diagrams used to represent elements.
\end{abstract}

\maketitle

\section*{Introduction}
Thompson's groups form interesting examples of a range of unusual
group-theoretic phenomena.  Their metric properties are understood
to varying extents of completeness.  A natural notion of ``size'' of
an element in these groups is the size in terms of the number of
nodes in the smallest tree pair diagram which can represent that
element. Fordham \cite{blakegd} developed an effective method for
computing the word length of elements of Thompson's group $F$
exactly from tree pair representations, which has led to good
understanding of both fine-scale and large-scale geometry of the
Cayley graph of $F$. Burillo, Cleary, Stein and Taback \cite{thompt}
give estimates for the word length in Thompson's group $T$, showing
that the word length is quasi-isometric to the number of nodes in
the smallest representative of a group element.    This has led to
some understanding of the large-scale geometry of $T$. The metric
properties of Thompson's group $V$ are less well-understood.  Birget
\cite{birget} gives lower and upper bounds in terms of the number of
nodes of minimal representatives-- the lower bound is linear in the
number of nodes, and the upper bound is proportional to $n \log{n}$,
where $n$ is the number of nodes in the minimal representative. This
upper bound is sharp in the sense that the fraction of elements
which have lengths close to the $n \log n$ bound converges quickly
to 1 as $n$ increases.  So there is a gap between the lower and
upper bounds which cannot be bridged merely by looking at the size
of representative diagrams.  Here, we consider the metric properties
of the braided Thompson's groups $BV$ and $\widehat{BV}$. These
groups have a great deal in common with Thompson's groups $V$ and
can be regarded as extensions of $V$.  Furthermore, they have some
properties in common with the finitely generated braid groups, whose
metric properties with respect to the standard generators are
well-understood in terms of the number of crossings in a minimal
diagram.

Below, we give upper and lower bounds for the word lengths of
elements of braided Thompson's groups $BV$ and $\widehat{BV}$ with
respect to their standard generating sets.  We give bounds based on
the number of nodes and the number of crossings, and we give some
sharper bounds based on the number of crossings of a single pair of
strands. We give examples of families of elements where the number of
crossings is significantly larger than the length, and use these examples
to show that the bounds on word length are optimal in terms of order of growth.

\section{Background on braided Thompson's groups}

Thompson's group $V$  is constructed
as a group of piecewise linear maps of the interval $[0,1]$, not
necessarily continuous, but which are linear in open intervals bounded by
dyadic rationals.  Cannon, Floyd and Parry \cite{cfp} give an excellent
introduction with complete details.
Elements in the group $V$
can be seen as pairs of rooted binary trees, with a permutation defining how
leaves are mapped to each other.
 The related group $\widehat{V}$ is
a subgroup and a supergroup of $V$ described in Brin \cite{brinbv1} which has
some aspects which make presentation and analysis more straightforward.

The standard ``Artinification" construction that replaces permutations
by braids is used to construct the braided versions $BV$ and
$\widehat{BV}$. See Brin \cite{brinbv1} and Dehornoy \cite{dehornoy} for details
of the constructions.

We construct a braided tree pair diagram representative in ``tree-braid-tree" form of an element of $BV$
as a pair of rooted binary trees, each with $n$ nodes, and a
braid with $n+1$ strands, thought of as joining the leaves of both
trees via the given braiding, with the second tree positioned upside down. See Figure \ref{bvelem} for
an example of an element of $BV$ in tree-braid-tree form.

\begin{figure}
\includegraphics[width=2in]{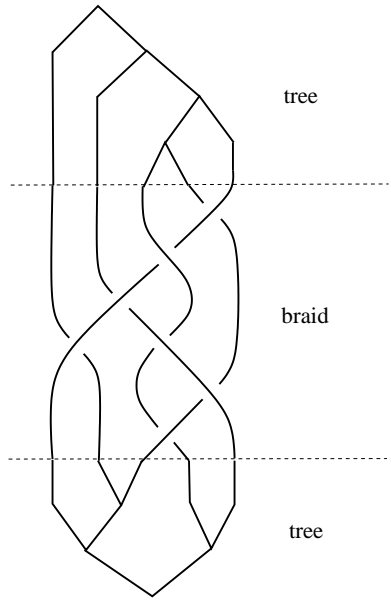}\\
\caption{A diagram for an element in $BV$}
\label{bvelem}
\end{figure}

An element in $BV$ is an equivalence class of diagrams, where the
equivalence is given by replacing two corresponding leaves of the
two trees by nodes, and splitting the corresponding strand into two
parallel strands.
 Using this
operation and its inverse operation of collapsing parallel strands which have
immediate common parent nodes in both trees, we obtain elements as
equivalence classes of braided tree pair diagrams.
There is a unique minimal size representative in each equivalence class.
 We multiply elements
by concatenation of two appropriately chosen tree-braid-tree diagram representatives.
By successive splittings of strands, we make the bottom tree of the first representative coincide with the top
tree of the second representative, and then we
cancel those trees and concatenate the resulting braids. An
example of multiplying elements of $BV$ is shown in Figure \ref{bvmult}. The interested reader can see many explicit details and examples in Brady, Burillo, Cleary and Stein
\cite{bbcs}, as well as the work of Brin \cite{brinbv1} and
Dehornoy \cite{dehornoy}.

\begin{figure}
\includegraphics[width=4in]{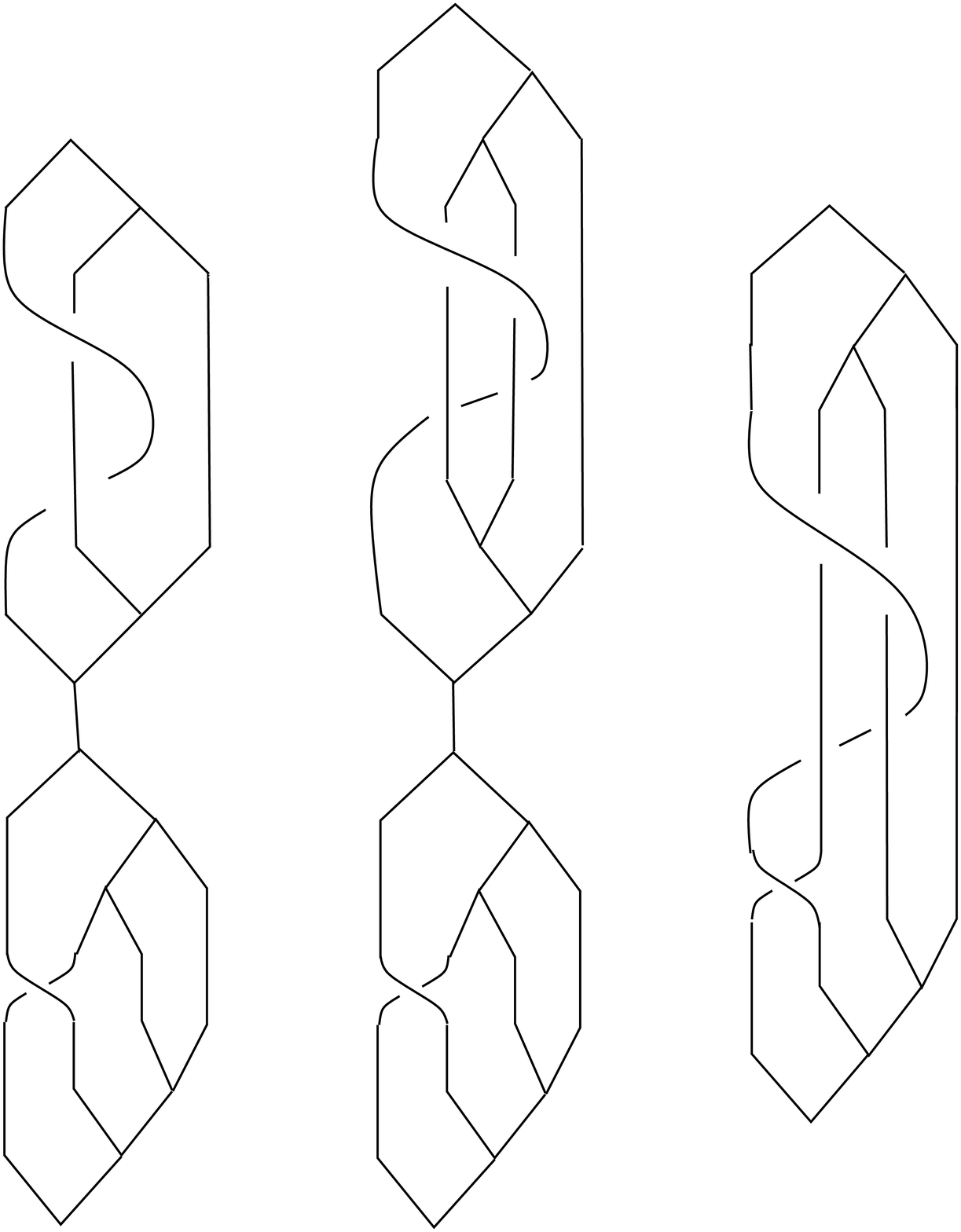}\\
\caption{An example of multiplication in $BV$}
\label{bvmult}
\end{figure}

We can consider, as done in \cite{bbcs}, the group
$\widehat{BV}$ as the subgroup of $BV$ that has the last strand
unbraided. Though our results below are all stated for $BV$, there are
immediate natural analogous results for $\widehat{BV}$ which simply omit the generators $\tau_i$ and any relations involving $\tau_i$. Presentations for both groups are given in \cite{bbcs}; here is a presentation for $BV$.

\begin{thm}[Brady, Burillo, Cleary, Stein \cite{bbcs}]
\label{bvinfpres}$BV$ has generators
 $x_i$,  $\sigma_i$, $\tau_i$, and relators
\begin{itemize}
\item [(A)]$x_j x_i=x_i x_{j+1} \text{  for }j > i$ \item
[(B1)]$\sigma_i \sigma_j=\sigma_j \sigma_i \text{  for }j-i \geq 2$
\item [(B2)]$\sigma_i \sigma_{i+1} \sigma_i = \sigma_{i+1} \sigma_i
\sigma_{i+1}$ \item [(B3)]$\sigma_i \tau_j = \tau_j \sigma_i\text{
for }j-i \geq 2$ \item [(B4)]$\sigma_i \tau_{i+1} \sigma_i =
\tau_{i+1} \sigma_i \tau_{i+1}$. \item [(C1)]$\sigma_i
x_j=x_j\sigma_i,\text{ for  }i<j$ \item [(C2)]$\sigma_i x_i=x_{i-1}
\sigma_{i+1} \sigma_{i}$ \item [(C3)]$\sigma_i x_j=x_j
\sigma_{i+1},\text{ for  }i \geq j+2$ \item
[(C4)]$\sigma_{i+1}x_i=x_{i+1}\sigma_{i+1}\sigma_{i+2}$ \item
[(D1)]$\tau_i x_j = x_j \tau_{i+1},\text{ for  }i-j \geq 2$ \item
[(D2)]$\tau_i x_{i-1}= \sigma_i \tau_{i+1}$ \item
[(D3)]$\tau_i=x_{i-1}\tau_{i+1}\sigma_i$.
\end{itemize}
\end{thm}

In this presentation, the generators $x_i$ correspond to the
generators of Thompson's group $F$ (see \cite{cfp}). The generators
$\sigma_i$ and $\tau_i$ are the braid generators, and all of them are
given by tree-pair-tree diagrams where the braid is a one-crossing braids
and the trees are all-right trees (trees which only have right child nodes.) The braided generators are:

\begin{itemize}
\item generators $\sigma_i$: crosses strand $i$ over strand $i+1$ between all-right trees
with $i+2$ leaves,
\item generators $\tau_i$: crosses strand $i$ over strand $i+1$ between all-right trees
with $i+1$ leaves.
\end{itemize}

The subgroup $\widehat{BV}$ is generated by all generators which have the
last strand unbraided; that is, given by words  $x_i$ and $\sigma_i$.

These presentations are infinite, but both groups admit finite
presentations.  We use  finite generating sets for understanding the metrics
on these groups. The group $BV$ is generated by $x_0$, $x_1$,
$\sigma_1$ and $\tau_1$, and its subgroup $\widehat{BV}$ is
generated by  $x_0$, $x_1$, and
$\sigma_1$.  Figure \ref{gens} shows the standard four generators.

\begin{figure}
\includegraphics[width=4.5in]{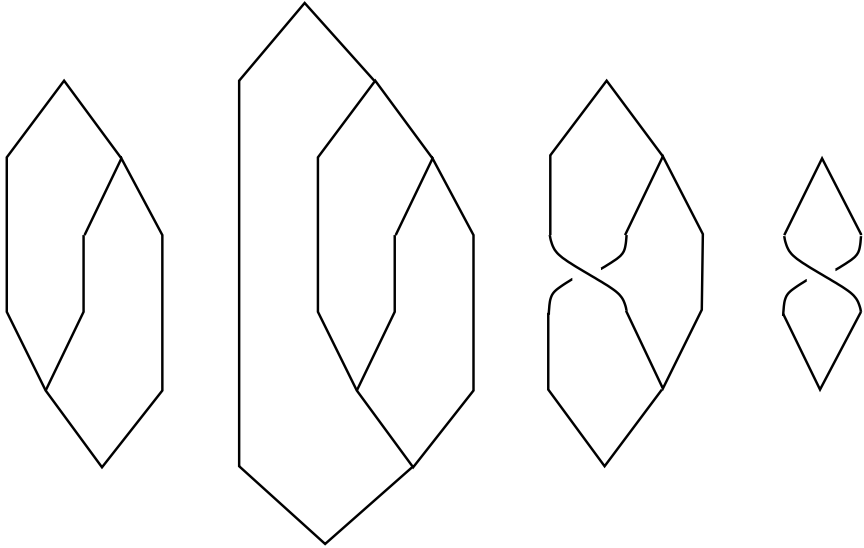}\\
\caption{The generators of $BV$, from left to right, $x_0$, $x_1$,
$\sigma_1$, $\tau_1$} \label{gens}
\end{figure}

A word in the generators of $BV$ corresponds potentially
to the concatenation of many tree-braid-tree diagrams, but it can always be
rewritten, using the relators, as a word which corresponds to a
single tree-braid-tree diagram, see \cite{bbcs}. Such a word would
be of the type
$$
(x_{i_1}^{r_1}x_{i_2}^{r_2}\ldots x_{i_n}^{r_n})(
Br)(x_{j_m}^{-s_m}\ldots x_{j_2}^{-s_2}x_{j_1}^{-s_1})
$$
where $Br$ is a braid, expressed as product of generators $\sigma_i$
and $\tau_i$. The minimal (reduced) representative diagram as a word in the
generators for an
element, will be used to compute and to estimate the metric. Of
course, the minimal length word for an element may likely not
correspond to a diagram given in tree-braid-tree form.

For each $n$, the ordinary $n$-strand braid groups $B_n$, generated
by the crossing of adjacent strands (usually denoted $\sigma_i$),
naturally are all subgroups of $BV$ merely by considering the
$n$-strands as the leaves of all-right trees.

\section{Motivating examples.}

\begin{figure}
\includegraphics[width=6.3in,height=8.3in]{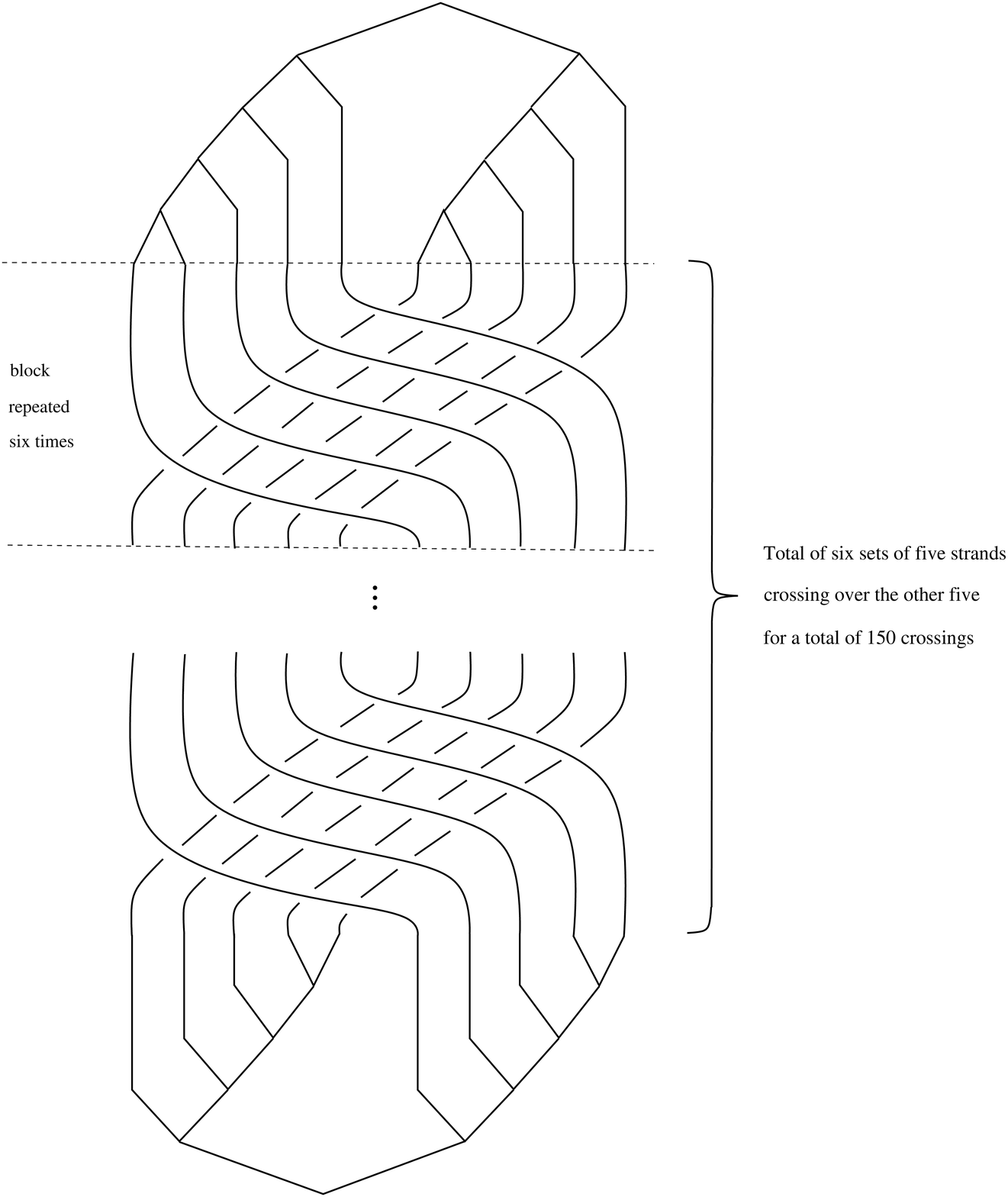}\\
\caption{The element $(x_1\tau_1)^6$, with word length 12, with 9
nodes and 6 blocks where the first 5 strands go over the last 5
strands, giving a total of 150 crossings.} \label{xtaupicture}
\end{figure}

In Thompson's group $F$, whose elements are pairs of
trees, word length is  proportional to the number of nodes in the trees.
In the ordinary $n$-strand braid groups, word length with respect to the
standard $n-1$ generators is given by the number of crossings in a minimal
braid diagram.  Thus, it is reasonable to suspect in our setting,
where the elements are combinations of trees and
braids, the number of nodes and number of crossings
have to be considered. The main difference is that here, if the  relevant trees
are not of the same shape, we may need
to split strands to be able to multiply elements, and splitting strands will
increase the total number of crossings.  The following examples show that
there can be interactions which rule out the possibility that word length
is estimated simply as being proportional to the number of strands and
the number of crossings.

\begin{example} \label{x1tau1} The reduced tree-braid-tree diagram
for the element $(x_1\tau_1)^{2n}$ has $2n+3$ nodes and $2n(n+2)^2$
crossings.
\end{example}

This follows from a standard induction argument on $n$. For $n=1$ we see
5 nodes and 18 crossings. It is not hard to check the values for a
general $n$ once we see the particular shape of these elements-- the interleaved
splitting and braiding quickly leads to diagrams with many crossings. See
Figure \ref{xtaupicture} for the example of the element $(x_1\tau_1)^6$, which
has word length 12, and a total of 150 crossings

Another family of examples of interesting elements is given by the
families of  Garside elements, which are half-twists of all strands.
These elements and their powers can play a central role in various
normal forms for ordinary braids, see Birman and Brendle
\cite{birmanbrendle}. The Garside element $\Delta_n= \sigma_1
\sigma_2 \cdots \sigma_n \sigma_1 \cdots \sigma_{n-1} \cdots
\sigma_1 \sigma_2 \sigma_1$ is the longest positive permutation
braid in $B_n$ and the left-multiple by $x_0^{-3}$ of the Garside element $\Delta_4$ is shown as an
element of $BV$ in Figure \ref{garsidepic}.  Each strand crosses
every other strand exactly once, giving ${n \choose 2}$ crossings
and thus the word length in $B_n$ is quadratic in the number of
strands, even when the number of generators grows linearly with the
number of strands.  Direct substitution of those expressions into
generators of $BV$ to find word length in $BV$ is far from optimal,
though, as shown by this example.

\begin{example} \label{garside}
The  Garside element $\Delta_{n+1}$ in the standard $B_{n+1}$ subgroup
of $BV$ has length at most $6n-7$ in $BV$ for $n\geq 2$.
\end{example}

This is again seen easily.  The Garside element $\Delta_2$ in $B_2$ is exactly $\tau_1$, and we see that products of the form $\tau_1 \tau_2 \cdots \tau_n= x_0^{n-1} \Delta_{n+1}$ giving that $\Delta_{n+1}=   x_0^{1-n} \tau_1 \cdots \tau_n$.
To substitute in terms of the standard finite generating set, first we note that
$\tau_n = x_0^{2-n} \tau_2 x_0^{n-2}$ for $n \geq 2$ and then that $\tau_2= x_0^{-1} \tau_1 \sigma_1^{-1}$.
Substituting to express $\Delta_{n+1}$ in terms of the finite generating set, we have  $\Delta_{n+1} = x_0^{1-n} \tau_1 \tau_2 (x_0^{-1} \tau_2 x_0)  (x_0^{-2} \tau_2 x_0^2) \cdots  (x_0^{-n+2} \tau_2 x_0^{n-2})$ which has length $6n-7$ when $\tau_2$ is expressed in terms of the 4 generators, for $n\geq 2$.

So in the $n-1$ generator subgroup $B_n$ of $BV$, the length of $\Delta_n$ grows
quadratically with $n$, yet in $BV$ itself, with only 4 generators, we can obtain the Garside element $\Delta_n$ as the  product of generators whose length grows at most linearly with the number of generators.

\begin{figure}
\includegraphics[width=3.5in]{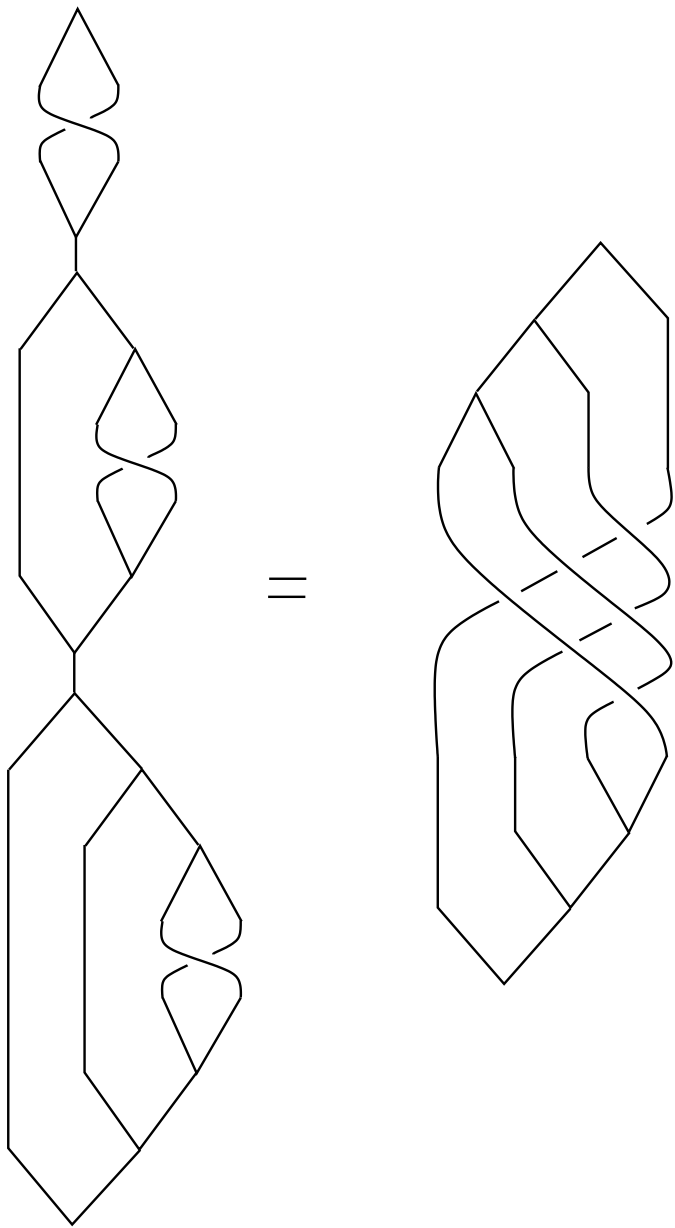}\\
\caption{Constructing the Garside element $\Delta_4$ as a product of $\tau_i$.
Here we show $\tau_1 \tau_2 \tau_3 = x_0^{-3} \Delta_4$.} \label{garsidepic}
\end{figure}

\section{Metric properties}

The examples in the previous section indicate that overall,
 the number of crossings is not a
good indicator of the length, since there are elements with large
number of crossings and small length. But the number of crossings
can still give an upper bound on the length.


\begin{thm}
Given an element  $w$ in $BV$ or $\widehat{BV}$  which has a reduced
tree-pair-tree diagram representative with $n$ nodes and $k$ crossings then there
exists a constant $C$ such that the length $|w|$  is at most
$C(n+nk)$.
\end{thm}

\begin{proof}
We consider then the reduced tree-braid-tree diagram of the element
$w$. It can be split into three diagrams using all-right trees,
giving a split into a positive element of $F$, a braid on two
all-right trees, and a negative element of $F$. The elements of $F$
have length at most proportional to $n$ (see \cite{bcs}). For the middle
braid in all-right trees, and with $k$ crossings, we can split it
into $k$ elements, each one of them with the same all-right tree on
top and bottom, and one single crossing. Each one of these elements
is then a representative for a generator $\sigma_i$ or $\tau_i$, for
any $i\ge 1$. We only need to go back to the expression of each
$\sigma_i$ and $\tau_i$ in terms of the finite set of generators
$x_0,x_1,\sigma_1,\tau_1$, and verify that, if $i\le n$, each
$\sigma_i$ or $\tau_i$ can be written with a number of generators
which is at most linear in $n$.  Thus substitution gives the stated upper bound on word length.
\end{proof}

 Examples \ref{x1tau1} and \ref{garside} from the previous section
show that the number of crossings may grow cubically or
quadratically with word length. So to find lower bounds, we need to
consider quantities  in the diagram which are more closely related
to the length than the number of crossings.  These quantities are
the number of nodes and the maximum number of crossings between any
pair of strands.

We let $w$ be an element given in a reduced tree-braid-tree diagram representative, and
let $n$ be the number of nodes. We let $s_{ij}$, for
$i,j=1,2,\ldots,n+1$ be the number of times the $i$-th and $j$-th
strands cross each other, and we set the maximum number of crossings between
any pair of strands as
$$
s=\max\{s_{ij}\,|\,i,j=1,2,\ldots,n+1\}.
$$

First, we remark that the number of nodes in a reduced diagram gives a
lower bound for the word length.

\begin{prop} Given an element  $w$ in $BV$ or $\widehat{BV}$  which
has a reduced diagram representative with $n$ nodes then there
exists a constant $C$ such that the length $|w|$  is at least $Cn$.
\end{prop}

This follows by standard arguments analogous to those
used for other groups of the Thompson
class. In the case of $BV$,
multiplication by any of the generators can add at most two nodes
to a reduced tree braid tree diagram, so an element with $n$ nodes will
require at least $\frac{n}2$ generators.

Furthermore, the maximal
number of crossings for a pair of strands also gives a lower bound
for the metric.

\begin{prop} \label{strandcross} Given an element  $w$ in $BV$ or $\widehat{BV}$
which has a reduced tree-braid-tree diagram representative where $s$ is the maximal
number of crossings of any pair of strands. Then the length $|w|$
is at least $s$.
\end{prop}

The proof becomes straightforward once we observe that
representatives of the generators $x_0$ and $x_1$ never have any
crossings, and representatives for the braid generators are
restricted by the following lemma.

\begin{lemma} \label{taumaxs}
In any tree-braid-tree representative of $\tau_1$ or $\sigma_1$ with $n$ nodes, any
pair of strands cross only once. That is, for any representative of
$\tau_1$ or $\sigma_1$, we have $s=1$.
\end{lemma}
\begin{proof}
The reduced representatives for $\sigma_1$ and $\tau_1$
have one single crossing. We obtain other representatives  from
this one by splitting strands. But if a pair of strands
cross only once, after the process of splitting any strand, in the
resulting braid, any given pair of strands still cross only once.
\end{proof}

Hence, Proposition \ref{strandcross} follows since representatives for $x_0$ and
$x_1$ have no crossings, and each time we multiply by $\tau_1$ or
$\sigma_1$, the corresponding generator we use will add at most one
crossing to each pair of strands.

There is an elementary relation between the total number of
crossings $k$ and  $s$, the maximal number of crossings of any pair of
strands, which is given multiplying by the number of pairs, ${{n+1}
\choose 2}$. Hence, we have that
$$
k\le \frac{n(n+1)}2 s.
$$

This leads to a lower bound on word length in terms of the number of crossings as follows:

\begin{prop} \label{cubicbound} There exists a constant $C$ such that if an
element  $w$ of $BV$ or $\widehat{BV}$ has length $N$, the number of crossings on its reduced tree-braid-tree
diagram representative satisfies $k\le CN^3$.
\end{prop}

\begin{proof} We consider what happens when we multiply
by a generator. The number of nodes $n$ grows by at most 2, since a
generator has at most 3 nodes and every element has at least the
root node. And in view of Lemma \ref{taumaxs}, $s$ can grow by only
1. Hence both $n$ and $s$ grow linearly with length, and hence, the
inequality linking  the total number of crossings $k$ to them
finishes the proof.
\end{proof}

Example \ref{x1tau1}, where the total number of crossings grows cubically with word length, shows that the cubic bound in Proposition \ref{cubicbound} is optimal.

From here we deduce these final results for the bounds, where we give
lower and upper bounds for the metric based on the quantities $k$ and
$s$.

\begin{thm}
For an element $w$ of $BV$ or $\widehat{BV}$ in tree-braid-tree form with $n$ nodes, $k$ total crossings, and with the maximum number of crossings of a pair of strands is $s$, there are constants
$C_1, C_2,C_3,C_4$ for which the metric satisfies the following
inequalities:
$$
C_1\max\{n,\sqrt[3]k\}\le C_2\max\{n,s\}\le |w|\le C_3(n+nk)\le
C_4(n+n^3s)
$$
\end{thm}

\bibliographystyle{plain}

\end{document}